\documentclass[11pt,leqno]{amsart}
\usepackage{amsmath,amssymb,amsthm}
\usepackage{hyperref}

\usepackage{bbm}

\DeclareMathOperator{\diam}{diam\,}
\DeclareMathOperator{\co}{co}

\renewcommand{\geq}{\geqslant}
\renewcommand{\leq}{\leqslant}

\newcommand{\supp}{\operatorname{supp}}
\newcommand{\spann}{\operatorname{span}}

\newtheorem{theorem}{Theorem}[section]
\newtheorem{lemma}[theorem]{Lemma}

\newtheorem{proposition}[theorem]{Proposition}
\newtheorem{corollary}[theorem]{Corollary}
\theoremstyle{definition}

\theoremstyle{remark}
\newtheorem{remark}[theorem]{Remark}
\numberwithin{equation}{section}

\def\fnote#1{\footnote}

\def\ignora#1{}
\def\n3#1{\left\vert  \! \left\vert \! \left\vert \, #1 \, \right\vert \!
  \right\vert \! \right\vert }

\begin{document}

\title{ Diameter, radius and Daugavet index of thickness of slices in Banach spaces }

\author{ Abraham Rueda Zoca }\address{Universidad de Granada, Facultad de Ciencias. Departamento de An\'{a}lisis Matem\'{a}tico, 18071-Granada
(Spain)} \email{ abrahamrueda@ugr.es}
\urladdr{\url{https://arzenglish.wordpress.com}}


\dedicatory{Dedicated to Luis Rueda Mesa, \emph{in memoriam}}

\maketitle

\markboth{ABRAHAM RUEDA ZOCA}{DIAMETER, RADIUS AND DAUGAVET INDEX OF THICKNESS OF SLICES}

\begin{abstract}
We construct a Banach space $X$ with the r-BSP such that the infimum of the diameter of the slices of the unit ball is $1$, which gives negative answer to a 2006 question by Y. Ivakhno in an extreme way. This example is performed by considering modifications of the classical James-tree space $JT_\infty$ constructed on a tree with infinitely many branching points $T_\infty$. Moreover we prove that every Banach space with the Daugavet property admits, for every $\varepsilon>0$,  an equivalent renorming for which its Daugavet index of thickness is bigger than $2-\varepsilon$ and there are slices of the unit ball of diameter strictly smaller than $2$, which solves an open question from \cite{hllnr}.
\end{abstract}

\section{Introduction}

The study of geometrical and topological properties of slices, weakly open sets and convex combinations of slices has attracted the attention of many researchers in Functional Analysis because they have determined multiple properties of Banach spaces. In connection with the existence of such objects of small diameter we can highlight the characterisations of the \textit{Radon-Nikodym property (RNP), the (convex) point of continuity property ((C)PCP)} or the \textit{strong regularity}. In the opposite extreme, the study of big slices, weakly open sets and convex combinations of slices have been analysed in connection with diameter two properties, octahedrality of the norm and the Daugavet property.

Very recently, the following indices were defined in \cite{rueda18,hllnr} with the aim of measuring how far a Banach space $X$ is from having the Daugavet property:

\begin{equation*}
  \mathcal{T}^s(X) = \inf\left\lbrace r>0 \;\middle|\;
  \begin{tabular}{@{}l@{}}
    \text{ there exist $x\in S_X$ and a slice $S$ of $B_X$}\\ \text{ such that} $S \subset B(x,r)$
   \end{tabular}
  \right\rbrace,
\end{equation*}

\begin{equation*}
  \mathcal{T}(X) = \inf\left\lbrace r>0 \;\middle|\;
  \begin{tabular}{@{}l@{}}
    \text{ there exist $x\in S_X$ and a relatively weakly }\\ \text{ open $W$ in $B_X$ such that} $\emptyset\neq W \subset B(x,r)$
   \end{tabular}
  \right\rbrace,
\end{equation*}

\begin{equation*}
  \mathcal{T}^{cc}(X) = \inf\left\lbrace r>0 \;\middle|\;
  \begin{tabular}{@{}l@{}}
    \text{ there exist $x\in S_X$ and a convex  }\\ \text{ combination of slices $C$ of $B_X$}\\ \text{  such that} $C \subset B(x,r)$
   \end{tabular}
  \right\rbrace.
\end{equation*}

The well known characterisations of the Daugavet property in terms of slices, weakly open sets and convex combinations of slices described in \cite{kssw,shv} read as follows: a Banach space $X$ has the Daugavet property if, and only if, any of the above indices is exactly 2 (in which case all of them coincide). The situation is dramatically different for intermediate values for these indices. In general, the equalities $0\leq \mathcal T^{cc}(X)\leq \mathcal T(X)\leq \mathcal T^s(X)\leq 2$ may be strict. Indeed, there are examples of Banach spaces $X$ for which $\mathcal T^s(X)\geq 1$ but $\mathcal T(X)=0$ \cite{hllnr}. In a similar way, there are examples of $X$ for which $\mathcal T(X)\geq 1$ and $\mathcal T^{cc}(X)=0$. 

There are immediate connections between the indices $\mathcal T^s(X)$ (resp. $\mathcal T(X)$, $\mathcal T^{cc}(X)$) and the infimum of the diameter of slices (resp. non-empty weakly open sets and convex combinations of slices). For instance, if every slice (resp. non-empty relatively weakly open subset, convex combination of slices) of $B_X$ has diameter 2 then $\mathcal T^s(X)\geq 1$ (resp. $\mathcal T(X)\geq 1, \mathcal T^{cc}(X)\geq 1$). 

The question whether the converse holds true was analysed in \cite[Section 3]{hllnr}. The question for the index $\mathcal T^{cc}$ was solved in a negative way in \cite[Example 3.6]{hllnr}. The question for the index $\mathcal T^s$ was also negatively solved, but its solution requires a bit more of attention. Closely related to the fact that $\mathcal T^s(X)\geq 1$ is the \textit{r-big slice property (r-BSP)} defined by Y. Ivakhno in \cite{ivakhno}: a Banach space $X$ is said to have the r-BSP if 
$$r(S):=\inf\{r>0: S\subseteq B(x,r)\mbox{ for some }x\in X\}$$
is greater than or equal to $1$ for every slice $S$ of $B_X$. It is clear that the r-BSP implies $\mathcal T^s(X)\geq 1$. Moreover, if every slice of $B_X$ has diameter 2 then $X$ has the r-BSP. Ivakhno asked whether the converse holds true, that is, whether the r-BSP implies that every slice of $B_X$ has diameter exactly 2. Ivakhno's question was solved in a negative way. In order to shorten notation, given a Banach space $X$, write
$$r(X):=\inf\{r(S): S\mbox{ is a slice of }B_X\};$$ 
$$d(X):=\inf\{\diam(S): S\mbox{ is a slice of }B_X\}.$$
Using the above notation, in \cite[Theorem 3.7]{hllnr} it was proved that $X=JT_\infty$ satisfies the r-BSP but, for every $\varepsilon>0$, there exists a slice of diameter $\leq \sqrt{2}+\varepsilon$; in other words, $r(X)=1<d(X)\leq \sqrt{2}$. Finally, the question whether $\mathcal T(X)\geq 1$ implies that every non-empty relatively weakly open subset of $B_X$ has diameter two remained open \cite[Remark 3.9]{hllnr}.

The aim of this note is to continue with the line of \cite{hllnr} in order to give a solution to two questions. The first one is, in view of the example $JT_\infty$, how small $d(X)$ can be in a Banach space $X$ with $r(X)\geq 1$? The second question we face in this note is precisely the question posed in \cite[Remark 3.9]{hllnr}.

With respect to the first question, the main result is the following.

\begin{theorem}\label{theo:counterbestia}
There exists a Banach space $X$ satisfying that $r(X)=d(X)=1$. 
\end{theorem}

In order to prove the result, motivated by \cite[Theorem 3.7]{hllnr}, we introduce in Section \ref{section:treespaces} a family of Banach spaces $JT_\infty^p$ for $1<p<\infty$, which are a variant of the space $JT_\infty$. We prove that these spaces are dual Banach spaces, whose preduals are denoted by $B_\infty^p$. We prove in Section \ref{section:predualesbuenpropi} that, for every $1<p<\infty$, $1\leq r(B_\infty^p)\leq d(B_\infty^p)\leq 2^\frac{1}{q}$, where $\frac{1}{p}+\frac{1}{q}=1$ (see Theorem \ref{theo:slicesbinfp}). Finally, the desired example $X$ is constructed as an infinite $\ell_1$-sum of these spaces.

We obtain as another consequence of Theorem \ref{theo:slicesbinfp} that, for every $\varepsilon>0$, there exists a Banach space $X$ with the CPCP and such that $d(X)>2-\varepsilon$, which is connected with the open question from \cite{blr15eje1} whether there exists a Banach space $X$ with the PCP and such that every slice of $B_X$ has diameter exactly $2$.

With respect to the second question, we devote Section \ref{section:renoweakopen} to proving the following theorem.

\begin{theorem}\label{theo:counternegadauga}
Let $X$ be a Banach space with the Daugavet property. Then, for every $\varepsilon>0$, there exists an equivalent renorming $\vert\cdot\vert_\varepsilon$ such that:
\begin{enumerate}
    \item $X$ fails the r-BSP and, in particular, its unit ball contain slices of diameter strictly smaller than $2$.
    \item $\mathcal T^{cc}(X)\geq 2-\varepsilon$.
\end{enumerate}
\end{theorem}

A particular consequence of the above theorem is that a Banach space $X$ may have $\mathcal T(X)\geq 1$ and yet containing non-empty weakly open subsets of $B_X$ of diameter strictly smaller than 2, which solves \cite[Remark 3.9]{hllnr}. Moreover, incidentally, we find an example of a Banach space $X$ for which $T^s(X)>1$ but with $B_X$ containing slices of radius strictly smaller than $1$. In other words, given a slice $S$ of $B_X$, the inequality
\[\begin{split}
\inf\{r>0: S\subseteq B(x,r)\mbox{ for some }x\in X\}\leq \\  \inf\{r>0: S\subseteq B(x,r)\mbox{ for some }x\in S_X\}
\end{split}\]
may be strict.

\bigskip

\textbf{Terminology:} We only consider real Banach spaces. The closed unit ball of a Banach space $X$ is denoted
by $B_X$ and its unit sphere by $S_X$. The dual space of $X$ is
denoted by $X^\ast$ and the bidual by $X^{\ast\ast}$.

By a \emph{slice} of $B_X$ we mean a set of the form
\begin{equation*}
S(B_X, x^*,\alpha) :=
\{
x \in B_X : x^*(x) > 1 - \alpha
\},
\end{equation*}
where $x^* \in S_{X^*}$ and $\alpha > 0$. If $X$ is a dual space, say $X=Y^*$ and $x\in Y$, the previous set is said to be a \textit{$w^*$-slice of $B_X$}. A finite convex combination of slices is a set of the form 
\[
\sum_{i=1}^n\lambda_i S(B_X, x^*_i, \alpha_i),
\]
where $n\in \mathbb N$ and $\lambda_i\in [0,1]$ such that $\sum_{i=1}^n\lambda_i=1$. 

We write ${\rm co}(D)$ (resp., $\overline{{\rm co}}(D)$) to denote the convex hull (resp., closed convex hull) of~$D$.

A Banach space $X$ is said to have the Daugavet property if every rank-one operator $T:X\longrightarrow X$ satisfies the equality
\begin{equation}\label{ecuadauga}
\Vert T+I\Vert=1+\Vert T\Vert,
\end{equation}
where $I$ denotes the identity operator on $X$. We refer the reader to \cite{kssw,shv,wer} and references therein for a detailed treatment on the Daugavet property. In Section \ref{section:renoweakopen} we will make use of the following characterisation of the Daugavet property: a Banach space $X$ has the Daugavet property if, and only if, for every $x\in S_X$, every convex combination of slices $C$ of $B_X$ and every $\varepsilon>0$ there exists $y\in C$ such that $\Vert x-y\Vert>2-\varepsilon$ (see the proof of \cite[Lemma 3]{shv}).

\section{Construction of the tree spaces}\label{section:treespaces}

We  begin with the construction of a family of Banach spaces which are a modification of the space $JT_\infty$, which we will call $JT_\infty^p$.

Let us define
$$T_\infty:=\{(\alpha_1,\ldots, \alpha_k)\ :\ k\in\mathbb N, \alpha_1,\ldots, \alpha_n\in\mathbb N\}\cup \{\emptyset\}.$$

Given $(\alpha_1,\ldots, \alpha_k),(\beta_1,\ldots, \beta_p)\in T_\infty$
we say that
$$(\alpha_1,\ldots, \alpha_k)\leq (\beta_1,\ldots, \beta_p)\Leftrightarrow \left\{\begin{array}{cc}
\vert (\alpha_1,\ldots, \alpha_k)\vert\leq \vert(\beta_1,\ldots, \beta_p)\vert  & \ \\
\alpha_i=\beta_i & \forall 1\leq i\leq k,
\end{array}\right. $$
 where  $\vert
(\alpha_1,\ldots, \alpha_n)\vert :=n$ and $\vert \emptyset\vert :=0$. This binary  relation defines a partial order on $T_\infty$. Given an element $(\alpha_1,\ldots, \alpha_k)\in T_\infty$, the elements $(\alpha_1,\ldots, \alpha_k, n), n\in\mathbb N$ are called \textit{successors of $(\alpha_1,\ldots, \alpha_k)$}.

A segment in $T_\infty$ is a totally ordered and finite subset $S\subseteq T_\infty$.

Set $1<p<\infty$ and, given a finitely supported function $x:T_\infty\longrightarrow \mathbb R$ (i.e. $\supp(x):=\{t\in T_\infty\ |\ x(t)\neq 0\}$ is finite), let us consider
$$\Vert x\Vert=\sup\left( \sum_{i=1}^n  \left\vert \sum_{t\in S_i}x(t)\right\vert^p \right)^\frac{1}{p},$$
where the sup is taken over all families $\{S_1,\ldots,S_n\}$ of disjoint segments of $T_\infty$.

Now we define $JT_\infty^p$ as the completion of  the space of finitely nonzero functions defined on $T_\infty$  for the above norm. Observe that, for $p=2$, the space is the classical $JT_\infty$ defined in \cite{goma}.

Up to our knowledge, the spaces $JT_\infty^p$ have not been previously considered in the literature for $p\neq 2$. Let us point out, however, that such family of spaces has been considered when working with binary trees, see e.g. \cite{bpv}.

Given
$t\in T_\infty$ let us define
$$e_t(s):=\left\{\begin{array}{cc}
1 & \mbox{if } s=t,\\
0 & \mbox{otherwise.}
\end{array} \right.$$

Then it is known that $\{e_t\}_{t\in T_\infty}$ is a (countable) Markusevic basis for $JT_\infty$ and that $JT_\infty$ is a dual space. We denote by $\{e_t^*\}_{t\in T_\infty}$ the biorthogonal functionals of $\{e_t\}_{t\in T_\infty}$ . Then $B_\infty := \overline{\spann}\{e_t^*\ :\ t\in T_\infty\}$, where the closure is taken in $JT_\infty^*$, is a complete predual of $JT_\infty$.

Our aim is to prove that the above also holds true for the space $JT_\infty^p$ for $1<p<\infty$. In order to do so, let us first introduce a bit of notation. Following \cite{goma}, a subset $A\subseteq T_\infty$ is \textit{full} if for every segment $S$ of $T_\infty$ the set $S\cap A$ is a segment. In this case, the projection $P_A:JT_\infty^p\longrightarrow JT_\infty^p$ given by 
$$P_A\left(\sum_{t\in T_\infty} x_t e_t \right):=\sum_{t\in A} x_t e_t$$
is a contraction. The adjoint operator $P_A^*$ defines a contraction from $B_\infty^p\longrightarrow B_\infty^p$.
Given $n\in\mathbb N$ define $L_n:=\{t\in T_\infty : \vert t\vert=n\}$. Denote by $P_n$ the projection defined by $L_n$. Given $n<m$ we denote by $P_n^m$ the projection defined by $L_n\cup L_{n+1}\cup\ldots\cup L_m$.

Let us observe that, given any finite set $F\subseteq T_\infty$, there exists another finite and full set $\tilde F\subseteq T_\infty$ containing $F$. Indeed, a simple cardinality argument shows that, since $F$ is finite, there exists $m\in\mathbb N$ such that, given any $t=(\alpha_1,\ldots, \alpha_k)\in F$, then $k\leq m$ and $\alpha_i\leq m$ holds for every $1\leq i\leq k$. It is not difficult to prove that
$$\tilde F:=\{(\beta_1,\ldots,\beta_m)\in T_\infty: \beta_i\leq m\mbox{ holds for all }1\leq i\leq m\}$$
is full, and clearly $F\subseteq \tilde F$. 

Using this simple observation we start by proving the following proposition.

\begin{proposition}\label{prop:acocompleta}
    Let $1<p<\infty$. Let $(a_t)_{t\in T_\infty}\subseteq \mathbb R$ such that
    $$\sup\left\{\left\Vert \sum_{t\in F} a_t e_t \right\Vert: F\subseteq T_\infty\mbox{ is finite and full} \right\}<\infty.$$
    Then the $\sum_{t\in T_\infty} a_t e_t$ converges.
\end{proposition}

\begin{proof} Assume for contradiction that $\sum_{t\in T_\infty} a_t e_t$ does not converge. Then, a negation of the Cauchy condition of summability, yields in particular a sequence of pairwise disjoint finite subsets $F_n\subseteq T_\infty$ and $\delta>0$ such that $\Vert \sum_{t\in F_n}a_t e_t\Vert>\delta$ for every $n\in\mathbb N$. Given $n\in\mathbb N$ find a family of disjoint segments $S_1^n,\ldots S_{k_n}^n$ in $T_\infty$ such that $\left(\sum_{i=1}^{k_n}\left\vert \sum_{t\in S_i^n} a_t\right\vert^p\right)^{\frac{1}{p}}>\delta$. Since $\supp (\sum_{t\in F_n}a_t e_t)\subseteq F_n$ we can assume with no loss of generality, up to cutting some segments, that $S_1^n,\ldots S_{k_n}^n$ are contained in $F_n$, which makes the family of segments $\{S_i^n: n\in\mathbb N, 1\leq i\leq k_n\}$ pairwise disjoint. We can now find a full and finite set $F$ containing $\cup_{i=1}^n F_i$. Now
\[\begin{split}
\left\Vert \sum_{t\in F} a_t e_t\right\Vert\geq \sum_{i=1}^n \left(\sum_{j=1}^{k_i} \left\vert \sum_{t\in S_j^i}a_t  \right\vert^p\right)^\frac{1}{p}\geq \sum_{i=1}^n \delta=n\delta,
\end{split}\]
which contradicts the fact that $\sup\left\{\left\Vert \sum_{t\in F} a_t e_t \right\Vert: F\subseteq T_\infty\mbox{ is finite and full} \right\}<\infty$ and proves the result.
\end{proof}

For consistency with the notation of \cite{goma} we write $B_\infty^p:=\overline{\spann}\{e_t^*: t\in T_\infty\}\subseteq (JT_\infty^p)^*$.

Also for consistency with the notation of \cite[Section 2]{blr15},  we define a molecule as a functional of the form
$$x^*:=\sum_{i=1}^n \lambda_i f_{S_i}$$
for $S_1,\ldots, S_n$ disjoint segments of $T_\infty$ and $\sum_{i=1}^n
\vert\lambda_i\vert^q= 1$ for $\frac{1}{p}+\frac{1}{q}=1$, where $f_S\in B_{(JT_\infty^p)^*}$ is defined by the equation
$$f_S(x):=\sum_{t\in S} x(t)$$
whenever $S\subseteq T_\infty$ is a segment of $T_\infty$.

Denote by $M$ the set of molecules in $(JT_\infty^p)^*$.

\begin{lemma}\label{normantejt}\
$M$ is a norming subset of $B_{(JT_\infty^p)^*}$. As a consequence
\begin{equation}\label{describoladualjt}
B_{(JT_\infty^p)^*}=\overline{\co}^{w^*}(M).
\end{equation}
\end{lemma}

\begin{proof}
Let $x^*=\sum_{i=1}^n \lambda_i f_{S_i}\in M$. Let us prove that $\Vert x^*\Vert\leq 1$. In order to do so, pick $x\in B_{JT_\infty^p}$. An application of H\"older inequality derives
\[
\begin{split}
x^*(x)=\sum_{i=1}^n \lambda_i f_{S_i}(x)\leq \sum_{i=1}^n \vert \lambda_i\vert \vert f_{S_i}(x)\vert& \leq \left(\sum_{i=1}^n \vert \lambda_i\vert^q \right)^\frac{1}{q}\left(\sum_{i=1}^n \vert f_{S_i}(x)\vert^p \right)^\frac{1}{p}\\
& \leq \left(\sum_{i=1}^n \left\vert \sum_{t\in S_i} x(t) \right\vert^p \right)^\frac{1}{p}\leq 1
\end{split}
\]
since $\Vert x\Vert\leq 1$ by assumptions. This proves that $\Vert x^*\Vert\leq 1$. On the other hand, take $\mu_i\in S_{\ell_p^n}$ such that $\sum_{i=1}^n \lambda_i\mu_i=1$. Now, taking $x:T_\infty\longrightarrow \mathbb R$ such that $\sum_{t\in S_i} x(t)=\mu_i$ it is immediate that $\Vert x\Vert=1$ and that $x^*(x)=1$, so $\Vert x^*\Vert=1$.

Now it is time to prove that $M$ is norming for $JT_\infty^p$. In order to do so, take $x\in JT_\infty^p$ and $\varepsilon>0$, and let us find $x^*\in M$ with $x^*(x)>1-\varepsilon$. Since $\Vert x\Vert=1$ there exist, by the definition of the norm, pairwise disjoint segments $S_1,\ldots, S_n$ in $T_\infty$ such that $\left(\sum_{i=1}^n \left\vert \sum_{t\in S_i} x(t)\right\vert^p\right)^\frac{1}{p}>1-\varepsilon$. Now find $(\lambda_1,\ldots, \lambda_n)\in S_{\ell_q^n}$ such that $\sum_{i=1}^n \lambda_i f_{S_i}(x)>1-\varepsilon$. Now set $x^*:=\sum_{i=1}^n \lambda_i f_{S_i}\in M$. It is clear that $x^*(x)>1-\varepsilon$, as desired. 

From a separation argument we get now that
$B_{(JT_\infty^p)^*}=\overline{\co}^{w^*}(M)$.
\end{proof}

\begin{remark}\label{rem:normamole}
    By the above, if $\lambda_1,\ldots, \lambda_n\in\mathbb R$ and $S_1,\ldots, S_n$ are pairwise disjoint segments of $T_\infty$ it follows
    $$\left\Vert \sum_{i=1}^n \lambda_i f_{S_i}\right\Vert=\left(\sum_{i=1}^n \vert \lambda_i\vert^q \right)^\frac{1}{q}.$$
\end{remark}

Now we are ready to prove the following desired result.

\begin{proposition}
Given $1<p<\infty$, the mapping $\phi:JT_p^\infty\longrightarrow (B_\infty^p)^*$ defined by $\phi(x)(e_t^*):=e_t^*(x)$ is an onto linear isometry.
\end{proposition}

\begin{proof}
It is immediate that $\phi$ is linear and continuous. The fact that $\phi$ is isometric follows from the fact that the set of all molecules, which is norming for $JT_\infty^p$ by Lemma \ref{normantejt}, is contained in $B_\infty^p$. 

It remains to prove that $\phi$ is onto, for which we will make use of Proposition \ref{prop:acocompleta}. In order to do so, take any functional $z\in (B_\infty^p)^*$. We define the function $x:T_\infty\longrightarrow \mathbb R$ defined by $x(t):=z(e_t^*)$ for every $t\in T_\infty$. We aim to prove that $x$ belongs to $JT_\infty^p$ because, once we have this proved, it is immediate that $\phi(x)=z$ and the proof would be finished. In order to prove that $x\in JT_\infty^p$ we will make use of Proposition \ref{prop:acocompleta}. To this end, take a finite and full set $F\subseteq T_\infty$ and take any molecule $\sum_{i=1}^n\lambda_i f_{S_i}\in M$. Now
$$\sum_{i=1}^n \lambda_i f_{S_i}\left(\sum_{t\in F}x(t)e_t\right)=\sum_{i=1}^n \lambda_i \sum_{t\in S_i\cap F} x(t)=\sum_{i=1}^n \lambda_i \sum_{t\in S_i\cap F} z(e_t^*).$$
Since $F$ is full we get that $S_i\cap F$ is a segment for every $1\leq i\leq n$, and clearly they form a collection of pairwise disjoint segments. Moreover observe that $\sum_{t\in S_i\cap F} z(e_t^*)=z(f_{S_i\cap F})$ by definition. Using the above equality together with the linearity of $z$ we infer
$$\sum_{i=1}^n \lambda_i f_{S_i}\left(\sum_{t\in F}x(t)e_t\right)=z\left(\sum_{i=1}^n \lambda_i f_{S_i\cap F}\right)\leq \Vert z\Vert,$$
where we have used that $\sum_{i=1}^n \lambda_i f_{S_i\cap F}$ is a molecule since $S_i\cap F$ is a segment for every $i$, that $M$ is contained in the unit sphere of $B_\infty^p$ and the continuity of $z$. 
Now Proposition \ref{normantejt} together with the arbitrariness of the molecule taken implies that
$$\left\Vert \sum_{t\in F} x(t)e_t \right\Vert=\sup_{\varphi\in M}\varphi\left( \sum_{t\in F} x(t)e_t\right)\leq \Vert z\Vert.$$
Proposition \ref{prop:acocompleta} implies that $\sum_{t\in T_\infty} x(t)e_t$ converges (and clearly converges to $x$). In particular, this implies that $x\in JT_\infty^p$. The fact that $\phi(x)=z$ is immediate now, which proves that $\phi$ is onto and finishes the proof.
\end{proof}

The rest of the section is devoted to showing that $B_\infty^p$ has the \textit{convex point of continuity property} for every $1<p<\infty$. Recall that a Banach space $X$ is said to have the \textit{point of continuity property (PCP)} (respectively the \textit{convex point of continuity property (CPCP)}) if every non-empty, bounded and closed (respectively every non-empty, bounded, closed and convex) subset of $B_X$ contains non-empty relatively weakly open subsets of arbitrarily small diameter. It is immediate that the PCP implies the CPCP. The converse is not true \cite{gms}.

The proof that $B_\infty^p$ has the CPCP is an adaptation of the proof of \cite[Theorem 2.2]{gms}, where it is proved that the predual of $JT_\infty$ has the CPCP.

Now we have the following result, which is a variant of \cite[Theorem IV.2]{goma}.

\begin{proposition}\label{prop:limpnpredual}
Given $1<p<\infty$ we have
$$B_\infty^p=\{y^*\in (JT_\infty^p)^*: \liminf \Vert P_n^*(y^*)\Vert_\infty=0\}.$$
where $\Vert\cdot\Vert_\infty$ stands for the sup-norm for a function defined on $T_\infty$.
\end{proposition}

\begin{proof}
The inclusion $\subseteq$ is clear, so let us prove $\supseteq$. In order to do so, assume that $y^*\in (JT_\infty^p)^*$ satisfies that $d(y^*, B_\infty^p)\geq \varepsilon$. A variation of the argument \cite[Lemma IV.2]{goma} allows to find a full subtree $T_1\subseteq T_\infty$ with finitely-many branching points such that $\Vert y^*-P_{T_1}^*(y^*)\Vert<\frac{\varepsilon}{2}$. Indeed, take $(\varepsilon_t)_{t\in T_\infty}$ positive numbers such that $\sum_{t\in T_\infty}\varepsilon_t<\frac{\varepsilon}{2}$. For every $t\in T_\infty$ call $S_t:=\{s\in T_\infty: s\geq t\mbox{ and } \vert s\vert=\vert t\vert+1\}$. A consequence of Remark \ref{rem:normamole} is that
$$\left\Vert \sum_{s\in S_t} P_{\{s\}}^*(x^*)\right\Vert^q=\sum_{s\in S_t} \Vert P_{\{s\}}^*(x^*) \Vert^q, $$
where $\frac{1}{p}+\frac{1}{q}=1$. In view of the above equality we can find a finite set $S_t^1$ such that
$$\sum_{s\in S_t\setminus S_t^1}\Vert P_{\{s\}}^*(x^*) \Vert^q=\left\Vert \sum_{s\in S_t\setminus S_t^1} P_{\{s\}}^*(x^*)\right\Vert^q<\varepsilon_t.$$
Now the construction of $T_1$ is clear. For every $t\in T$ consider only the sucessors of $T$ which belong to $S_t^1$, and repeat the argument with those elements. Observe that the terms eliminated from $x^*$ have norm smaller than $\sum_{t\in T_\infty}\varepsilon_t<\frac{\varepsilon}{2}$.

Now consider $x^*:=P_{T_1}^*(y^*)$. Observe that $d(x^*,B_1)\geq \frac{\varepsilon}{2}$, where $B_1:=\overline{\spann}\{e_t^*:t\in T_1\}$. We claim that there exists a branch $B\subseteq T_1$ (which is still a branch in $T_\infty$) such that $\lim_{t\in B} x^*(e_t)\neq 0$. Indeed, set $\Gamma$ to be the set of all the branches in $T_1$. It follows that the operator $S:(JT_1)_p'\longrightarrow \mathbb R^\Gamma$ defined by
$$S(\phi):=\left(\lim_{t\in \gamma} \phi(e_t) \right)_{\gamma\in \Gamma}$$
is a well defined operator \cite[pp. 11]{bpv} whose kernel equals $B_1$ \cite[pp. 12]{bpv}. Since $x^*(e_t)=y^*(e_t)$ for every $t\in T_1$ we infer that $y^*\notin \{y^*\in (JT_\infty^p)^*: \liminf \Vert P_n^*(y^*)\Vert_\infty=0\}$, and the proof is finished.
\end{proof}

The following corollary is a variant of \cite[Lemma 1.2]{gms}, whose proof is immediate from the above proposition.

\begin{corollary}\label{coro:finalproof}
Let $y^*\in (JT_\infty^p)^*$ satisfying that
$$\liminf_n \Vert P_n^*(y^*)\Vert_\infty=0.$$
Then there is not any increasing sequence $(n_k)_k$ and $\alpha>0$ such that
$$\left\Vert (P_{n_k+1}^{n_{k+1}})^* (y^*)\right\Vert_\infty\geq \alpha.$$
\end{corollary}

In order to prove that $B_\infty^p$ has the CPCP the following lemma is needed.

\begin{lemma}\label{lemma:cpcp}
    Set $1<p<\infty$ and $C$ be a closed convex subset of $B_\infty^p$, $M\in\mathbb N$ and $\varepsilon>0$. There exists a relatively weakly open convex subset $U$ of $C$ and $N>M$ such that 
    $$\Vert P_N(U)\Vert_\infty=\sup_{y\in U} \Vert P_N(y)\Vert_\infty<\varepsilon.$$
\end{lemma}
\begin{proof}
    The proof is verbatim that of \cite[Lemma 2.1]{gms} choosing $m>1+(\frac{2n}{\varepsilon})^{q}\frac{1}{n}$, for $\frac{1}{p}+\frac{1}{q}=1$, and using that, in the notation of that proof, we have
    $$\left\Vert P_N(y_0^m) \right\Vert^{q}=\sum_{j=1}^{m-1}\sum_{i=1}^n \vert y_0^m(t_i^j)\vert^{q}.$$
\end{proof}

Now the following result holds.

\begin{theorem}\label{theo:cpcp}
    For every $1<p<\infty$ the space $B_\infty^p$ has the CPCP.
\end{theorem}

\begin{proof}
    The proof follows following word-by-word the proof of \cite[Theorem 2.2]{gms}, taking into account that $(P_n^m)^*(B_\infty)$ is isomorphic to $\ell_{q}$, for $\frac{1}{p}+\frac{1}{q}=1$, which has the CPCP by the reflexivity. At the end of the proof entail a contradiction with Corollary \ref{coro:finalproof}.
\end{proof}

\section{Slices in $JT_\infty^p$}\label{section:predualesbuenpropi}

The aim of this section is to describe the geometric properties of the slices of $B_{B_\infty^p}$. Let us prove the following.

\begin{theorem}\label{theo:slicesbinfp} Let $1<p<\infty$ and let $q$ such that $\frac{1}{p}+\frac{1}{q}=1$. The following holds.
\begin{enumerate}
    \item $B_\infty^p$ has the r-BSP.
    \item Every slice of $B_\infty^p$ has diameter greater than or equal to $2^\frac{1}{q}$.
    \item For every $\varepsilon>0$ there exists a slice of $B_\infty^p$ such that $\diam(S)\leq 2^\frac{1}{q}+\varepsilon.$
\end{enumerate}
\end{theorem}

For the proof we will make use of the following proposition, which relates the diameter and the radius of a slice $S(B_X,f,\alpha)$ with that of the weak-star slice $S(B_{X^{**}},f,\alpha)$. This will be useful in our context in order to work in $(JT_\infty^p)^*=(B_\infty^p)^{**}$, where we will be able to take advantage of Lemma \ref{normantejt}. The proof relies on an easy argument using Goldstine's theorem and the weak$^*$ lower semi-continuity of the norm.

\begin{proposition}\label{prop:propibidual}
    Let $X$ be a Banach space, $f\in S_{X^*}$ and $\alpha>0$. The following assertion hold:
    \begin{enumerate}
        \item $\diam(S(B_X,f,\alpha))=\diam(S(B_{X^{**}},f,\alpha))$.
        \item $r(S(B_X,f,\alpha))\geq r(S(B_{X^{**}},f,\alpha))$.
    \end{enumerate}
\end{proposition}

Now we are able to provide the pending proof.

\begin{proof}[Proof of Theorem \ref{theo:slicesbinfp}]
In view of Proposition \ref{prop:propibidual} we will prove that $(B_\infty^p)^{**}=(JT_\infty^p)^*$ satisfies that every weak-star slice has radius $1$, diameter bigger than or equal to $2^{\frac{1}{q}}$ and satisfying that, for every $\varepsilon>0$, there exists a weak-star slice of the unit ball of diameter smaller than $2^\frac{1}{q}+\varepsilon$.

Let us start by proving (1). Let $x^*\in (JT_\infty^*)^*$, $\varepsilon>0$ and a $w^*$-slice $S=S(B_{(JT_\infty^p)^*}, x,\alpha)$, and let us find $y^*\in S$ such that $\Vert x^*-y^*\Vert\geq 1-\varepsilon$. We can assume by a density argument that $x$ is finitely supported. Since $M$ is norming we get $M\cap S\neq \emptyset$, in other words, there exists $g:=\sum_{i=1}^n \lambda_i f_{S_i}\in M$ with $g\in S$. Find $(\mu_1,\ldots, \mu_n)\in \ell_p^n$ such that $\sum_{i=1}^n \lambda_i \mu_i=1$. For every $i\in\{1,\ldots, n\}$ let $t_i\in S_i$ be the node of maximal level. Fix an arbitrary $i$. Observe that, if we write $(u_k)_{k\in\mathbb N}$ the colection of succesors of $t_i$, it is clear that, for every $j\in\mathbb N$ and any collection $\{\alpha_1,\ldots, \alpha_j\}\subseteq \mathbb R$ then
$$\left\Vert \sum_{i=1}^j \alpha_i e_{u_i}\right\Vert=\left(\sum_{i=1}^j \vert \alpha_i\vert^p \right)^\frac{1}{p}.$$
This in particular implies that, if we call $(f_k)$ the canonical basis of $\ell_p$, the mapping $\phi:\ell_p\longrightarrow JT_\infty^p$ such that $\phi(f_k)=e_{u_k}$, for all $k\in\mathbb N$, is an isometry and, in particular, it is weakly null because of the weak-to-weak continuity of the above operator. Consequently, we can find a successor $u_n$ of $t_i$ such that $x^*(e_{u_n})$ is as small as we wish and such that $u_n\notin \supp(x)$. This proves that, for every $i$, we can find a successor $u_i$ of $t_i$ large enough to get $x^*(\sum_{i=1}^n \mu_i e_{u_i})<\varepsilon$ and such that $x(u_i)=0$ holds for every $i$, from where we conclude that $y^*:=\sum_{i=1}^n \lambda_i f_{S_i\cup\{u_i\}}\in S$. Now consider $z:=\sum_{i=1}^n \mu_i e_{u_i}$, and it is clear that $\Vert z\Vert=1$ since $\{u_1,\ldots, u_n\}$ are respective successors of elements which were uncomparable. Consequently
$$\Vert y^*-x^*\Vert\geq y^*(z)-x^*(z)=\sum_{i=1}^n \lambda_i \mu_i-\varepsilon=1-\varepsilon.$$
This proves (1).

(2) Take a $w^*$-slice $S$ of $B_{(JT_\infty^p)^*}$. As before we can find $g=\sum_{i=1}^n \lambda_i f_{S_i}\in M$ with $g\in S$. We can assume that $S=S(B_{(JT_\infty^p)^*}, x,\alpha)$ and that $x$ is finitely supported. Call $t_i\in S_i$ the element of maximal level. Observe that, since there are infinitely many successors of $t_i$ and $x$ is finitely support we can find, for every $i$, two different successors $u_i$ and $v_i$ of $t_i$ such that the elements $x^*=\sum_{i=1}^n \lambda_i f_{S_i\cup\{u_i\}}, y^*=\sum_{i=1}^n \lambda_i f_{S_i\cup\{v_i\}}$ belong to $S$. Our aim is now to prove that $\Vert x^*-y^*\Vert\geq 2^\frac{1}{q}$. In order to do so, find $(\mu_1,\ldots, \mu_n)\in S_{\ell_p^n}$ such that $\sum_{i=1}^n \lambda_i\mu_i=1$ and define $z:=\sum_{i=1}^n \mu_i(e_{u_i}-e_{v_i})$. Observe that, since $\{u_1,v_1,\ldots, u_n,v_n\}$ are different eachother since the segments $S_1,\ldots, S_n$ were pairwise disjoint we derive that $\Vert z\Vert\leq \left(\sum_{i=1}^n 2\vert \mu_i\vert^p\right)^\frac{1}{p}=2^\frac{1}{p}$. Consequently we have
$$\Vert x^*-y^*\Vert 2^\frac{1}{p}\geq (x^*-y^*)(z)=\sum_{i=1}^n \lambda_i (\mu_i+\mu_i)=2\sum_{i=1}^n \lambda_i\mu_i=2,$$
so $\Vert x^*-y^*\Vert\geq 2^{1-\frac{1}{p}}=2^\frac{1}{q}$.

Finally let us prove (3). In order to do so, let $\alpha>0$, and consider $x^*,y^*\in S(B_{(JT_\infty^p)^*},e_\emptyset, \alpha)\cap M$, and let us estimate $\Vert x^*-y^*\Vert$. In order to do so, we write $x^*=\sum_{i=1}^n \lambda_i f_{S_i}$
 and $y^*=\sum_{j=1}^m \mu_j f_{T_i}$, for suitable families of pairwise disjoint segments $\{S_1,\ldots, S_n\}$ and $\{T_1,\ldots, T_m\}$ of $T_\infty$. Note that there exists just one $S_i$ such that $\emptyset\in S_i$ (we can assume with no loss of generality $i=1$). Analogously, assume $\emptyset\in T_1$ and $\emptyset\notin T_j$ for $j\geq 2$. Now
 $$1-\alpha<x^*(e_{\emptyset})=\lambda_1.$$
 Since $\sum_{i=1}^n \vert \lambda_i\vert^q\leq 1$ we infer that $\sum_{i=2}^n \vert \lambda_i\vert^q<1-\vert\lambda_1\vert^p\leq 1-(1-\alpha)^q$. In a similar way $\mu_1>1-\alpha$ and thus $\sum_{j=2}^m \vert \mu_i\vert^q<1-(1-\alpha)^q$.

Consequently $\Vert x^*-f_{S_1}\Vert\leq \vert 1-\lambda_1\vert \Vert f_{S_1}\Vert+\left\Vert \sum_{i=2}^n \lambda_i f_{S_i}\right\Vert\leq \alpha+\left(\sum_{i=2}^n \vert\lambda_i\vert^q \right)^\frac{1}{q}\leq \alpha+ (1-(1-\alpha)^q)^\frac{1}{q}$. Similarly we have $\Vert y^*-f_{T_1}\Vert\leq \alpha+ (1-(1-\alpha)^q)^\frac{1}{q}$.

Now let us estimate $\Vert f_{S_1}-f_{T_1}\Vert$. Both are segments that contain $\emptyset$, so $S_i=S_1\cap T_1\cup S_i\setminus T_i$ for $i=1,2$. This means that $f_{S_i}=f_{S_1\cap T_1}+f_{S_i\setminus T_i}$ and henceforth $f_{S_1}-f_{T_1}=f_{S_1\setminus T_1}-f_{T_1\setminus S_1}$. Since $S_1\setminus T_1$ and $T_1\setminus S_1$ are disjoint segments we get that $\Vert f_{S_1}-f_{T_1}\Vert=\Vert f_{S_1\setminus T_1}-f_{T_1\setminus S_1}\Vert=2^\frac{1}{q}$. Thus
$$\Vert x^*-y^*\Vert\leq 2^\frac{1}{q}+2(\alpha+(1-(1-\alpha)^q)^\frac{1}{q}).$$
Since $2(\alpha+(1-(1-\alpha)^q)^\frac{1}{q})\rightarrow 0$ when $\alpha\rightarrow 0$ and by \cite[Theorem 1.1]{blr15} we have that $\inf_{\alpha>0} \diam(S(B_{(JT_\infty^p)^*},e_\emptyset,\alpha)=2^\frac{1}{q}$, and the theorem is proved. \end{proof}

The above theorem proves that, for every $\varepsilon>0$, there exists a Banach space $X$ such that $1=r(X)\leq d(X)\leq 1+\varepsilon$. Now we are ready to prove Theorem \ref{theo:counterbestia} by providing a Banach space $X$ satisfying that $r(X)=1=d(X)$, which is a negative answer to the open question by Ivakhno in \cite{ivakhno} in an extreme way.

\begin{proof}[Proof of Theorem \ref{theo:counterbestia}]
By Theorem \ref{theo:slicesbinfp} we have that, for every $n\in\mathbb N$, there exists a Banach space $X_n$ satisfying that $1\leq r(X_n)\leq d(X_n)\leq 1+\frac{1}{n}$ holds for every $n\in\mathbb N$. Write
$$X:=(\oplus_{n=1}^\infty X_n)_1.$$
We have that $r(X)\geq 1$ by \cite[Theorem 3]{ivakhno}.

In order to prove that $d(X)=1$ take $\varepsilon>0$ and let us find a slice $S$ of $B_X$ with $\diam(S)<1+\varepsilon$. In order to do so, pick $n\in\mathbb N$ with $\frac{1}{n}<\frac{\varepsilon}{4}$. Now there exists a slice $S(B_{X_n}, x_n^*,\alpha)$ of diameter smaller than or equal to $1+\frac{\varepsilon}{2}$. Since slices decrease with the index, there is no loss of generality in assuming that $\alpha<\frac{\varepsilon}{4}$. Now consider the slice
$$S=S(B_X, (\underbrace{0,0,...,0}_{n-1},x_n^*,0,0,\ldots ),\alpha).$$
Let us prove that $\diam(S)\leq 1+\varepsilon$. To do so, take $(x_k),(y_k)\in S$. We have $1-\alpha<x_n^*(x_n)\leq \Vert x_n\Vert$. This means, on the one hand, that $x_n\in S(B_{X_n},x_n^*,\alpha)$. On the other hand, since $1\geq \Vert (x_k)\Vert=\sum_{k=1}^\infty \Vert x_k\Vert$ we infer $\sum_{k\neq n} \Vert x_k\Vert<\alpha$. In a similar way, we have that $y_n\in S(B_{X_n}, x_n^*,\alpha)$ and that $\sum_{k\neq n} \Vert y_k\Vert<\alpha$. Now $\Vert x_n-y_n\Vert\leq 1+\frac{\varepsilon}{2}$ because $\diam(S(B_{X_n}, x_n^*,\alpha))<1+\frac{\varepsilon}{2}$. Then
\[\begin{split}
\Vert (x_k)-(y_k)\Vert=\sum_{k=1}^\infty \Vert x_k-y_k\Vert& =\Vert x_n-y_n\Vert+\sum_{k\neq n}\Vert x_k-y_k\Vert\\
\ & \leq 1+\frac{\varepsilon}{2}+\sum_{k\neq n}\Vert x_k\Vert+\sum_{k\neq n}\Vert y_k\Vert\\
& \leq 1+\frac{\varepsilon}{2}+2\alpha.
\end{split}\]
Since $(x_k)$ and $(y_k)$ were arbitrary we get that $\diam(S)\leq 1+\frac{\varepsilon}{2}+2\alpha<1+\varepsilon$ and the proof is complete.
\end{proof}

Another consequence of the construction of the spaces is the following.

\begin{theorem}\label{theo:examCPCPslicegrand}
For every $\varepsilon>0$ there exists a Banach space $X$ with the CPCP such that every slice of the unit ball has diameter at least $2-\varepsilon$.
\end{theorem}

\begin{proof}
Given $\varepsilon>0$ select $X=B_\infty^p$ for $1<p<\infty$ with $2^\frac{1}{q}>2-\varepsilon$, where $\frac{1}{p}+\frac{1}{q}=1$. Then $X$ has the CPCP by Theorem \ref{theo:cpcp} but every slice of the unit ball has diameter at least $2^\frac{1}{q}>2-\varepsilon$ in virtue of Theorem \ref{theo:counterbestia}
\end{proof}

\begin{remark}
After the proof of Corollary 2.6 in \cite{blr15eje1} it is asked whether there exists a Banach space $X$ with the PCP and such that every slice of $B_X$ has diameter exactly 2. The same question replacing the PCP with CPCP remains open to the best of the author knowledge. 
\end{remark}

\section{Renormings of Daugavet spaces}\label{section:renoweakopen}

The aim of this section is to prove Theorem \ref{theo:counternegadauga}, which in particular gives a negative answer to \cite[Question (b)]{hllnr}. The whole section will be a description of the renorming technique.

Let $X$ be a Banach space with the Daugavet property. Select $x_0\in S_X$ and take $f\in S_{X^*}$ such that $f(x_0)=1$. Now take $\delta>0$ and define
$$B_\delta:=\overline{\co}(B_X\cup\{(1+\delta)x_0\}\cup\{-(1+\delta)x_0\}).$$
It is immediate that $B_X\subseteq B_\delta\subseteq (1+\delta)B_X$. Moreover, observe that $\sup_{z\in B_\delta} f(z)=1+\delta$. Indeed, given $z\in B_\delta$ we can find $x\in B_X$ and $\alpha,\beta,\gamma\in [0,1]$ with $\alpha+\beta+\gamma=1$ and such that $z=\alpha x+\beta(1+\delta) x_0-\gamma (1+\delta) x_0$. Now
\[
\begin{split}
    f(z)=\alpha f(x)+\beta f((1+\delta)x_0)-\gamma f((1+\delta)x_0)& \leq \alpha+(\beta-\gamma)(1+\delta)\\
    & \leq \alpha+ \beta(1+\delta)\leq 1+\delta\\
    & =f((1+\delta)x_0).
\end{split}
\]
Now we have the following result.

\begin{proposition}\label{prop:diamslicesrenormingdauga}
Given $0<r<\delta$ it follows:
$$\diam(\{z\in B_\delta: f(z)>1+\delta-r\})\leq \frac{3r(1+\delta)}{\delta}.$$
\begin{proof}
Let $z,z'\in \{z\in B_\delta: f(z)>1+\delta-r\}$ and, as before, find $x,x'\in B_X$ and $\alpha, \alpha',\beta,\beta',\gamma,\gamma'\in [0,1]$ with $\alpha+\beta+\gamma=\alpha'+\beta'+\gamma'=1$ and such that $z=\alpha x+(\beta-\gamma)(1+\delta)x_0$ and that $z'=\alpha' x'+(\beta'-\gamma')(1+\delta)x_0$.
Since $f(z)>1+\delta-r$ we infer
\[
\begin{split}
1+\delta-r<f(z)=\alpha f(x)+(\beta-\gamma)(1+\delta)\leq 1+\beta \delta.
\end{split}
\]
Thus $\beta>1-\frac{r}{\delta}$. Since $\alpha+\beta+\gamma=1$ it is clear that $\alpha+\gamma<\frac{r}{\delta}$. In a similar way it is proved that $\beta'>1-\frac{r}{\delta}$ and consequently $\alpha'+\gamma'<\frac{r}{\delta}$. Hence we get $\vert \beta-\beta'\vert<\frac{r}{\delta}$ and then
\[\begin{split}
\Vert z-z'\Vert& \leq \vert \beta-\beta'\vert (1+\delta)\Vert x_0\Vert+ \alpha \Vert x\Vert+\alpha'\Vert x'\Vert+(\gamma+\gamma')(1+\delta)\\
& \leq \frac{r}{\delta}(1+\delta)+(1+\delta)(\alpha+\gamma+\alpha'+\gamma')\leq 3\frac{r}{\delta}(1+\delta),
\end{split}\]
and the proof is finished by the arbitrariness of $z$ and $z'$.
\end{proof}
\end{proposition}

Now given $\varepsilon>0$ consider
$$C_\varepsilon:=B_X+\varepsilon B_\delta.$$
$C_\varepsilon$ defines an equivalent norm by the Minkowski functional associated to $C_\varepsilon$, i.e.,
$$\vert x\vert_\varepsilon:=\inf\{r>0: x\in r C_\varepsilon\}.$$
The chain of inclusions $B_X\subseteq C_\varepsilon\subseteq B_X+\varepsilon(1+\delta)B_X=(1+\varepsilon(1+\delta))B_X$ implies
$$\frac{\Vert x\Vert}{1+\varepsilon(1+\delta)}\leq \vert x\vert_\varepsilon\leq \Vert x\Vert.$$
Let us prove that $(X,\vert\cdot\vert_\varepsilon)$ is the desired renorming described in Theorem \ref{theo:counternegadauga}. Let us start by proving that $(X,\vert \cdot\vert_\varepsilon)$ fails the r-BSP. In order to do so, let us start by the following observation:
$$\sup_{z\in C_\varepsilon} f(z)=1+\varepsilon(1+\delta).$$
Indeed, given $z\in C_\varepsilon$ we can find $x\in B_X$ and $y\in B_\delta$ such that $z=x+\varepsilon y$. Now
$$f(z)=f(x)+\varepsilon f(y)\leq 1+\varepsilon\sup_{v\in B_\delta} f(v)=1+\varepsilon(1+\delta).$$
On the other hand $(1+\varepsilon(1+\delta))x_0=x_0+\varepsilon((1+\delta) x_0)\in B_X+\varepsilon B_\delta=C_\varepsilon$ and $f((1+\varepsilon(1+\delta))x_0)=1+\varepsilon(1+\delta)$.

Now we have the following result.

\begin{proposition}\label{prop:slicerenormingduaganot2}
Given $0<r<\min\{\varepsilon,\delta\}$, the following inclusion holds
$$\{z\in B_{(X,\vert\cdot\vert_\varepsilon)}: f(z)>1+\varepsilon(1+\delta)-r\}\subseteq B\left(\varepsilon(1+\delta)x_0, \frac{1}{1+\varepsilon}+\frac{3r(1+\delta)}{\delta}\right).$$
In particular, $(X,\vert\cdot\vert_\varepsilon)$ fails the r-BSP and, consequently, the unit ball contain slices of diameter strictly smaller than 2.
\end{proposition}

\begin{proof}
Let $0<r<\min\{\varepsilon,\delta\}$ and take $z\in B_{(X,\vert\cdot\vert_\varepsilon)}=C_\varepsilon$ with $f(z)>1+\varepsilon(1+\delta)-r$. Hence, we can write $z=x+\varepsilon y$. Now
$$1+\varepsilon(1+\delta)-r<f(z)=f(x)+\varepsilon f(y)\leq 1+\varepsilon f(y),$$
from where $f(y)>1+\delta-\frac{r}{\varepsilon}$. Now Proposition \ref{prop:diamslicesrenormingdauga} applies to get that $\Vert y-(1+\delta)x_0\Vert\leq \frac{3\frac{r}{\varepsilon}(1+\delta)}{\delta}=r \frac{3(1+\delta)}{\delta\varepsilon}$. Hence
\[\begin{split}\vert z-\varepsilon(1+\delta)x_0\vert_\varepsilon\leq \vert x\vert_\varepsilon+\varepsilon \vert y-(1+\delta)x_0\vert_\varepsilon& \leq \vert x\vert _\varepsilon+\varepsilon \Vert y-(1+\delta)x_0\Vert\\
& \leq \vert x\vert_\varepsilon+ r \frac{3(1+\delta)}{\delta}.
\end{split}\]
Now it is time to estimate $\vert x\vert_\varepsilon$. In order to do so observe that $x\in B_X$ clearly. Hence
$$(1+\varepsilon)x=x+\varepsilon x\in B_X+\varepsilon B_X\subseteq B_X+\varepsilon B_\delta=C_\varepsilon.$$
This implies that $x\in \frac{1}{1+\varepsilon}C_\varepsilon$. Since $\vert x\vert_\varepsilon=\inf\{r>0: x\in r C_\varepsilon\}$ it is plain that $\vert x\vert_\varepsilon\leq \frac{1}{1+\varepsilon}$. Putting all together we get
$$\vert z-\varepsilon(1+\delta)x_0\vert_\varepsilon\leq \frac{1}{1+\varepsilon}+r\frac{3(1+\delta)}{\delta}.$$
From the arbitrariness of $z\in C_\varepsilon$ together with $f(z)>1+\varepsilon(1+\delta)-r$ we conclude that
$$\{z\in B_{(X,\vert\cdot\vert_\varepsilon)}: f(z)>1+\varepsilon(1+\delta)-r\}\subseteq B\left(\varepsilon(1+\delta)x_0, \frac{1}{1+\varepsilon}+\frac{3r(1+\delta)}{\delta}\right).$$
Hence, taking $r$ small enough, $X$ fails the r-BSP. Moreover, the above inclusion yields 
$$\diam(\{z\in B_{(X,\vert\cdot\vert_\varepsilon)}: f(z)>1+\varepsilon(1+\delta)-r\})\leq \frac{2}{1+\varepsilon}+\frac{6r(1+\delta)}{\delta}.$$
Taking infimum on $r$ the theorem is finished.\end{proof}

Now it is time to prove that $(X,\vert\cdot\vert_\varepsilon)$ has big index of thickness.

\begin{proposition}\label{prop:ccindexrenorming}
Given $x\in S_{(X,\vert\cdot\vert_\varepsilon)}$, $\eta>0$ and a convex combination of slices $C$ of $B_{(X,\vert\cdot\vert_\varepsilon)}$ there exists $y\in C$ such that 
$$\vert x-y\vert_\varepsilon\geq \frac{2-3\varepsilon(1+\delta)-\eta}{1+\varepsilon(1+\delta)}.$$
In particular, $T^{cc}((X,\vert\cdot\vert_\varepsilon)\geq \frac{2-3\varepsilon(1+\delta)}{1+\varepsilon(1+\delta)}$.
\end{proposition}

\begin{proof}
Write $C=\sum_{i=1}^n \lambda_i S(B_{(X,\vert\cdot\vert_\varepsilon)},f_i,\alpha_i)$ and select $\sum_{i=1}^n \lambda_i x_i\in C$ and $x\in S_{(X,\vert\cdot\vert_\varepsilon)}$. Now, for $1\leq i\leq n$, write $x_i=u_i+\varepsilon v_i$ with $u_i\in B_X$ and $v_i\in B_\delta$. Similarly, write $x=u+\varepsilon v$ with $u\in B_X$ and $v\in B_\delta$. Observe that
$$1=\vert x\vert_\varepsilon\leq \vert u\vert_\varepsilon+\varepsilon\vert v\vert_\varepsilon\leq \Vert x\Vert+\varepsilon(1+\delta)$$
since $v\in B_\delta\subseteq (1+\delta)B_X$. This implies $\Vert u\Vert>1-\varepsilon(1+\delta)$. 

On the other hand, given $1\leq i\leq n$, we have 
$$1-\alpha_i<f(x_i)=f(u_i)+\varepsilon f(v_i)\Rightarrow f(u_i)>1-\alpha_i-\varepsilon f_i(v_i).$$
Given $T_i:=\{z\in B_X: f_i(z)>1-\alpha_i-\varepsilon f_i(v_i)\}$, we get that $T_i$ is non-empty, and thus defines a slice of $B_X$.  Write $D:=\sum_{i=1}^n \lambda_i T_i$, which is a convex combination of slices of $B_X$. Let $\eta>0$. Since $X$ has the Daugavet property we can find, by the proof of \cite[Lemma 3]{shv}, elements $\hat u_i\in T_i$ for every $1\leq i\leq n$ such that
$$\left\Vert u-\sum_{i=1}^n \lambda_i \hat u_i\right\Vert>2-\varepsilon(1+\delta)-\eta.$$
Given $1\leq i\leq n$, since $\hat u_i\in T_i$ then $f_i(\hat u_i+\varepsilon v_i)>1-\alpha_i$. Since $\hat u_i+\varepsilon v_i\in B_X+\varepsilon B_\delta=C_\varepsilon$ we infer that $\sum_{i=1}^n \lambda_i (\hat u_i+\varepsilon v_i)\in C$. Now
\[
\begin{split}
\left\vert x-\sum_{i=1}^n \lambda_i (\hat u_i+\varepsilon v_i) \right\vert_\varepsilon& \geq \frac{\left\Vert x-\sum_{i=1}^n \lambda_i (\hat u_i+\varepsilon v_i) \right\Vert}{1+\varepsilon(1+\delta)}\\
& \geq \frac{\Vert u-\sum_{i=1}^n \lambda_i \hat u_i\Vert-\varepsilon\Vert  v-\sum_{i=1}^n \lambda_i v_i\Vert}{1+\varepsilon(1+\delta)} \\
& \geq \frac{2-\varepsilon(1+\delta)-\eta-2\varepsilon(1+\delta)}{1+\varepsilon(1+\delta)},
\end{split}
\]
where we have used that $v, \sum_{i=1}^n \lambda_i v_i\in B_\delta\subseteq (1+\delta)B_X$. This completes the proof. 
\end{proof}

\begin{proof}[Proof of Theorem \ref{theo:counternegadauga}]The proof follows by the construction of this section and the arbitrariness of $\delta$ and $\varepsilon>0$.\end{proof}

\begin{remark}\label{remark:solutionhllnr}
In \cite[pp. 18, Question (b)]{hllnr} it is asked whether $\mathcal T(X)\geq 1$ implies that every non-empty relatively weakly open subset of $B_X$ has diameter $2$. Theorem \ref{theo:counternegadauga} gives a negative answer in an extreme way: Given $\varepsilon>0$ there exists a Banach space $X$ with $\mathcal T(X)>2-\varepsilon$ and $X$ still contains slices of diameter strictly smaller than 2.   
\end{remark}

\begin{remark}\label{remark:thicknessyrBSP}
In \cite[pp. 18]{hllnr} it is pointed out that if $X$ has the r-BSP then $\mathcal T^s(X)\geq 1$. Proposition \ref{prop:slicerenormingduaganot2} proves that the converse does not hold. We find that, for every $\varepsilon>0$ there exists a Banach space $X$ with $\mathcal T^s(X)>2-\varepsilon$ but $X$ fails the r-BSP.
\end{remark}






\section*{Acknowledgements}  

The author wants to thank an anonymous referee for suggestions that improved the exposition. This work was supported by MCIN/AEI/10.13039/501100011033: Grant PID2021-122126NB-C31, Junta de Andaluc\'ia: Grants FQM-0185 and PY20\_00255, by Fundaci\'on S\'eneca: ACyT Regi\'on de Murcia grant 21955/PI/22 and by Generalitat Valenciana project CIGE/2022/97.

\end{document}